\documentclass[11pt]{amsart}

\usepackage{float}
\usepackage{graphicx} % Required for inserting images
\usepackage{mathtools}
\usepackage{amsfonts,amsmath,amssymb,amsthm}
\usepackage{latexsym}
\usepackage[normalem]{ulem} %L'option normalem change l'emphasize de la bibliographie par défaut, cf http://blog.dorian-depriester.fr/latex/bibtex-une-bibliographie-dans-les-regles-de-lart

\usepackage[a4paper]{geometry}
\usepackage{tikz}
\usepackage{xfrac}
\usepackage{enumitem}
\usepackage{marvosym}
\usepackage{todonotes}

\usepackage{hyperref}
\hypersetup{
    colorlinks=false,      % Use boxed links instead of colored
    linkbordercolor={0 1 0},  % Green boxes around internal links
    citebordercolor={0 1 0},  % Blue boxes around citations
    urlbordercolor={1 0 1}    % Magenta boxes around URLs
}

\usepackage{cleveref}

\theoremstyle{plain}
\newtheorem{theorem}{Theorem}[section]
\newtheorem*{theorem*}{Theorem}

\newtheorem*{proposition*}{Proposition}

\newtheorem{corollary}[theorem]{Corollary}

\theoremstyle{definition}
\newtheorem{definition}[theorem]{Definition}

\newtheorem{remark}[theorem]{Remark}

\newtheorem{maintheorem}{Theorem}

\newtheorem{maincorollary}[maintheorem]{Corollary}

\newcommand\restr[2]{{% we make the whole thing an ordinary symbol
  \left.\kern-\nulldelimiterspace % automatically resize the bar with \right
  #1 % the function
  \vphantom{\big|} % pretend it's a little taller at normal size
  \right|_{#2} % this is the delimiter
  }}

\usepackage{tikz,pgfplots}
\pgfdeclarelayer{bg}    % declare background layer
\pgfplotsset{compat=1.15}
\pgfsetlayers{bg, main} 
\usetikzlibrary{patterns} %for the oblique lines
\usetikzlibrary{positioning} %for loops
\usepackage{wrapfig}
\tikzset{c/.style={every coordinate/.try}} %to change spot of a coordinate

\newcommand{\zbimh}[4] % start, stop, bend big, bend small. Ex: \zbimh{(0,0)}{(1,0)}{80}{10}
{
\begin{pgfonlayer}{bg} 
\foreach \i in {1,...,7}
	{
	\draw #1 edge [bend left=#3, draw=none] node[pos=\i/8](a\i){}#2;
	\draw #1 edge [bend left=#4, draw=none] node[pos=\i/8](b\i ){}#2;
	}
\begin{scope}[every coordinate/.style={shift={(0,0.08*rand-0.08*rand)}}]
\filldraw[pattern=north west lines, pattern color=cyan!80!white, draw=cyan] plot [smooth] coordinates {#1 ([c]a1) ([c]a2) ([c]a3) ([c]a4) ([c]a5) ([c]a6) ([c]a7) #2} -- plot [smooth] coordinates {#2 ([c]b7) ([c]b6) ([c]b5) ([c]b4) ([c]b3) ([c]b2) ([c]b1) #1};
\end{scope}
\end{pgfonlayer}
}

\newcommand{\zbimv}[4] % départ, arrivée, tordage ext, tordage int. Ex: \zbimv{(0,0)}{(1,0)}{80}{10}
{
\begin{pgfonlayer}{bg} 
\foreach \i in {1,...,7}
{
\draw #1 edge [bend left=#3, draw=none] node[pos=\i/8](a\i){}#2;
\draw #1 edge [bend left=#4, draw=none] node[pos=\i/8](b\i ){}#2;
}
\begin{scope}[every coordinate/.style={shift={(0.08*rand-0.08*rand,0)}}]
\filldraw[pattern=north west lines, pattern color=cyan!80!white, draw=cyan] plot [smooth] coordinates {#1 ([c]a1) ([c]a2) ([c]a3) ([c]a4) ([c]a5) ([c]a6) ([c]a7) #2} -- plot [smooth] coordinates {#2 ([c]b7) ([c]b6) ([c]b5) ([c]b4) ([c]b3) ([c]b2) ([c]b1) #1};
\end{scope}
\end{pgfonlayer}
}

\newcommand{\zbimc}[4] % vertex, starting angle (0 is horizontal), ending angle, looseness Ex: \zbimc{(0,0)}{(1,0)}{80}{10}
{
{\tikzset{every loop/.style={in=#2,out=#3,looseness=#4}}
 \begin{pgfonlayer}{bg}
 \foreach \i in {1,...,5}
{
\draw #1 edge [loop, draw=none] node[pos=(\i-0.5)/5](a\i){} ();
}
\begin{scope}[every coordinate/.style={shift={(0.08*rand-0.08*rand,0.08*rand-0.08*rand)}}]
\filldraw[pattern=north west lines, pattern color=cyan!80!white, draw=cyan] plot [smooth] coordinates {#1 ([c]a1) ([c]a2) ([c]a3) ([c]a4) ([c]a5) #1};
\end{scope}
\end{pgfonlayer}
}
}

\usepackage{tikz}
\usetikzlibrary{positioning, automata}
\usetikzlibrary{decorations.markings} %for arrows
\usetikzlibrary{arrows.meta} %for arrow tips
\usetikzlibrary{calc} %to do calculations on coordinates

\newcommand{\vertex}{\node [circle, draw, inner sep=0pt, minimum size=4pt, fill=black]} %vertices

\title[Expansion properties of Whitehead moves]{Expansion properties of Whitehead moves \\ on cubic graphs}

\author{Laura Grave de Peralta} 
\address{Impact Initiatives, 9 Chemin de Balexert, 1219 Geneva, Switzerland}

\author{Alexander Kolpakov}
\address{University of Austin, 522 Congress Ave., Austin, TX 78701, United States}

\date{\today}

\begin{document}

\begin{abstract}
    The present note concerns the ``graph of graphs'' that has cubic graphs as vertices connected by edges represented by the so--called Whitehead moves. Here, we prove that the outer--conductance  of the graph of graphs tends to zero as the number of vertices tends to infinity. This answers a question of K.~Rafi in the negative. 
\end{abstract}

\maketitle

\setcounter{tocdepth}{1}
%\tableofcontents

%\listoftodos

%\todo[inline]{Still not sure about the title, but this we can change later. --Sasha}

\section{Introduction}\label{section-wm:intro}
The present note is concerned with the graph of connected $3$--regular, or cubic, graphs. Such a ``graph of graphs'' $\Gamma_n$, represents what can be informally described as the ``deformations space'' of cubic graphs on $2n$ vertices under the Whitehead moves. In particular, we shall investigate its expansion properties. 

%Expansion is an important notion for families of graphs: we wish to find explicit families of graphs with good expansion properties. In particular, one of the interests of studying expansion is the rate of convergence of a random walk to the stationary distribution that is controlled by the second largest eigenvalue of the transition matrix.  

Let $g$ be a cubic graph (we allow graphs with multiple edges and loops), with set of vertices $V(g)$, and set of edges $E(g)$. Then, for an edge $e \in E(g)$ which is not a loop, there are two possible Whitehead moves $w^{(1)}_e$ and $w^{(2)}_e$ on $e$, depicted in \Cref{wm-fig:Whitehead-move}, which are local transformations of $g$. 

\begin{figure}[ht]
\centering
\includegraphics[scale=0.35]{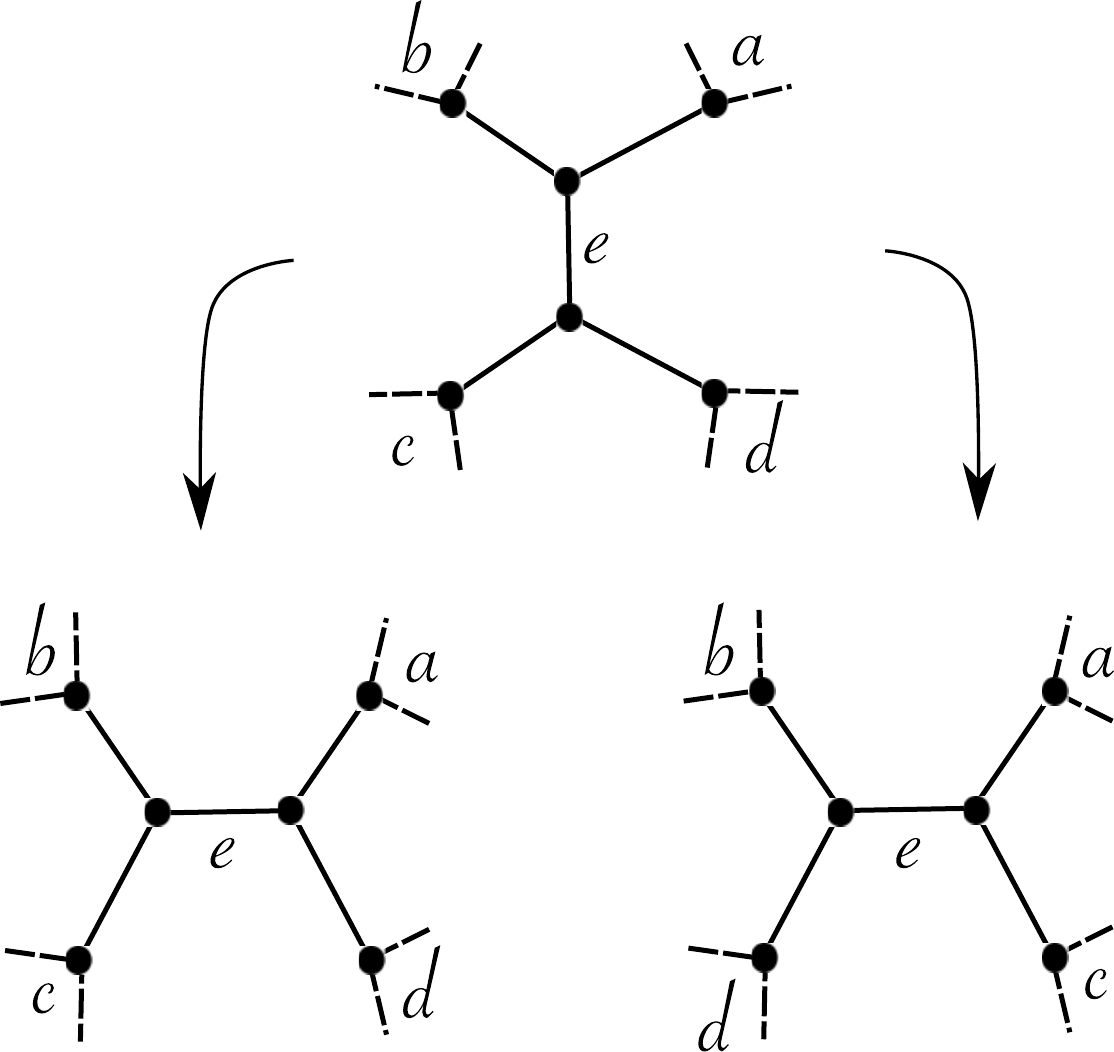}
\caption{ Two possible Whitehead moves performed on an edge $\varepsilon$. } \label{wm-fig:Whitehead-move}
\end{figure}

If $e \in E(g)$ is a loop, we define the respective Whitehead moves as having no effect on $g$. Let us note that this convention is not arbitrary. Indeed, we may represent $g$ as a multigraph: a set of half--edges incident to vertices and paired by a fixed--point--free involution. Then a Whitehead move does not affect the pairing of half--edges and only changes which vertices the half--edges are adjacent to. If $e$ is a loop, the two vertices
$u$, $v$ adjacent to its half--edges coincide: thus changing the vertex assignment of one of the other two half--edges at $u$ and of one of the other two half--edges at $v$ does nothing.

Let the vertices of $\Gamma_n$ be all possible cubic graphs on $2n$ vertices ($n\geq 1$) up to isomorphism, where two vertices $g_1, g_2 \in V(\Gamma)$ are connected by a directed edge for each Whitehead move $w_\varepsilon$, for $\varepsilon \in E(g_1)$, such that $w_\varepsilon(g_1)$ is isomorphic to $g_2$. It is easy to see that being related by a Whitehead move is an equivalence relation on the set of isomorphism classes of cubic graphs on $2n$ vertices ($n \geq 1$). 

In this note, a directed version of this graph of graphs is considered, as it is common to do so in the context of dynamical systems. Moreover, as Whitehead moves are clearly reversible, this choice does not affect the expansion property of $\Gamma_n$.

We introduce the notion of outer--conductance $\phi_{\rm out}$ because there is no standard definition of expansion for directed non--regular multigraphs, despite them being a natural class of objects to consider. 

Outer--conductance captures how easy it is to escape a given subset of the graph. Therefore, having the outer--conductance tend to $0$ is a way of saying that the graphs have poor mixing and therefore cannot be ``expanders'' in this sense. 

Namely, we show that the following theorem holds which answers the question about ``edge expansion'' of $\Gamma_n$ posed by K.~Rafi \cite{Rafiwebpage}. 

\begin{maintheorem}[\Cref{wm-thm:graph-of-graphs}]
Let $\Gamma_n$ be the graph of cubic graphs on $2n$ ($n\geq 1$) vertices connected by Whitehead moves. Then $\Gamma_n$ is connected and $\phi_{\rm out}(\Gamma_n) \rightarrow 0$, as $n\to \infty$.
\end{maintheorem}

Since $\Gamma_n$ contains loops, it is aperiodic. As a byproduct of this and Theorem~\Cref{wm-thm:graph-of-graphs}, we obtain one more statement about the combinatorics of $\Gamma_n$.

\begin{maincorollary}[\Cref{cor:graph-of-graphs-1}]
For each $n \geq 1$, the graph $\Gamma_n$ is strongly connected and aperiodic. There exists a Perron number $\rho_n$ such that the number of length $\ell$ paths in $\Gamma_n$ is asymptotic to ${\rm const} \cdot \rho^\ell_n$, as $\ell \to \infty$. 
\end{maincorollary}

Finally, in \Cref{wm-section:related}, we mention some objects connected to $\Gamma_n$.

First and foremost, the graph of graphs is related to moduli spaces of surface triangulations \cite{DP1, DP2} and pants decomposition of surfaces \cite{RafiTao}, although not entirely equivalent to said objects.
The definition considered here is a directed version of the graph of graphs considered in \cite{RafiTao}, and is quasi--isometric to the latter. 

Also, the graph of graphs considered in \cite{RafiTao} is quasi--isometric to the ``wide'' component of the thick part of the moduli space of a genus $g\geq 2$ surface. This means that $\Gamma_n$ describes pants decompositions of the genus $g = n+1$ surface in which the set of separating curves contains only sufficiently short geodesics. 

Similar local transformations are considered in \cite{HatcherThurston} and two such pants decompositions represent adjacent vertices in $\Gamma_n$ if they are connected by a $(I)$--move\footnote{Called an $A$--move in \cite{Hatcher}}, while a $(IV)$--move\footnote{Called an $S$--move in \cite{Hatcher}} is always performed on a handle that corresponds to a loop in the pants decomposition graph, and thus is not taken into account when passing to $\Gamma_n$. In this way, there exists a graph morphism from the directed pants graph to $\Gamma_n$ which sends pants decomposition to their dual graphs, and edges to edges. 

As shown in \cite{Wolf-preprint, Wolf-thesis}, the vertex set of $\Gamma_n$ coincides with the vertex set of the pants graph quotient by the action of the mapping class group. Thus, $\Gamma_n$ is quasi--isometric to the latter. 

Thus, knowing the combinatorial properties of $\Gamma_n$ could be useful in the study of the geometric and combinatorial properties of the pants graph.

%Here, we introduce a variation $\phi_{\rm out}$ of the usual notion of conductance and show that this quantity tends to $0$ for the family $\Gamma_n$ of cubic graphs on $2n$ vertices.

\section{Conductance and expansion}\label{section-wm:conductance-and-expansion}

Let $G$ be a connected directed graph with vertex set $V(G)$ and edge set $E(G)$. For a subset $S \subset V(G)$, let $d(S)$ denote the sum of vertex out--degrees in $S$, i.e. $d(S) = \sum_{v\in S} d_{\rm out}(v)$, where $d_{\rm out}(v)$ is the number of edges of the form $(v,u)$ where $u \in V(G)$. The \textit{boundary} $\partial(S)$ of a vertex subset $S \subset V(G)$ is defined as the set of directed edges in $E(G)$ joining a vertex in $S$ with a vertex in $V(G) \setminus S$, that is $\partial (S) = \{ e=(u,v) \in E(G) \ | \ u \in S \ , v \in S^c \} $. 

The \textit{outer--conductance} of $S \subset V(G)$ is defined as follows:
\[
\phi_{\rm out}(S) =  \frac{|\partial(S)|}{\min\{ d(S), d(V(G) - S)  \}}
\]
and the outer--conductance of $G$ is
\[
\phi_{\rm out}(G) = \min_{S \subset V(G)} \phi_{\rm out}(S).
\]

This is a generalization of the notion of conductance \cite[\S 6.2]{Chung}. Since we work with directed graphs, the definition is adapted so that the volume is measured with respect to the number of out--going edges. By doing so the outer--conductance measures how hard it is to escape a subset of the graph.

As a generalization of the notion of expander families to the case of directed graphs with vertex degrees (both out--degree and in--degree) growing with the number of vertices, we introduce \Cref{wm-def:outer_expander}.

\begin{definition}\label{wm-def:outer_expander}
We say that a sequence of directed graphs $G_n$ with out--degree and in--degree growing with the number of vertices is an \textit{outer--expander family} if the conductance $\phi_{\rm out}(G_n)$ is uniformly bounded away from 0 as $n$ tends to infinity. 
\end{definition}

Now we state our main result.  

\begin{theorem}\label{wm-thm:graph-of-graphs}
Let $\Gamma_n$ be the directed graph of cubic graphs on $2n$ ($n\geq 1$) vertices connected by Whitehead moves. Then $\Gamma_n$ is connected and $\phi_{\rm out}(\Gamma_n) \rightarrow 0$, as $n\to \infty$. Therefore $\{\Gamma_n\}_{n \in \mathbb{N}}$ is not an outer--expander family. 
\end{theorem}

\section{Proof of \texorpdfstring{\Cref{wm-thm:graph-of-graphs}}{Main theorem}} \label{section-wm:proof-theorem}

Let $\Gamma_n$ be the directed graph of (isomorphism classes of) cubic graphs on $2n$ vertices ($n \geq 1$) with edges corresponding to Whitehead moves as defined in the introduction. The following two claims will imply \Cref{wm-thm:graph-of-graphs}.

\subsection{$\Gamma_n$ is connected.} 

This result in mentioned in \cite{Rafiwebpage} as attributed to Coco Zhang. We give a proof below by using the classical facts about tree rotations and rebalancing known in computer science.   

\begin{figure}[ht]
\centering
\includegraphics[scale=0.3]{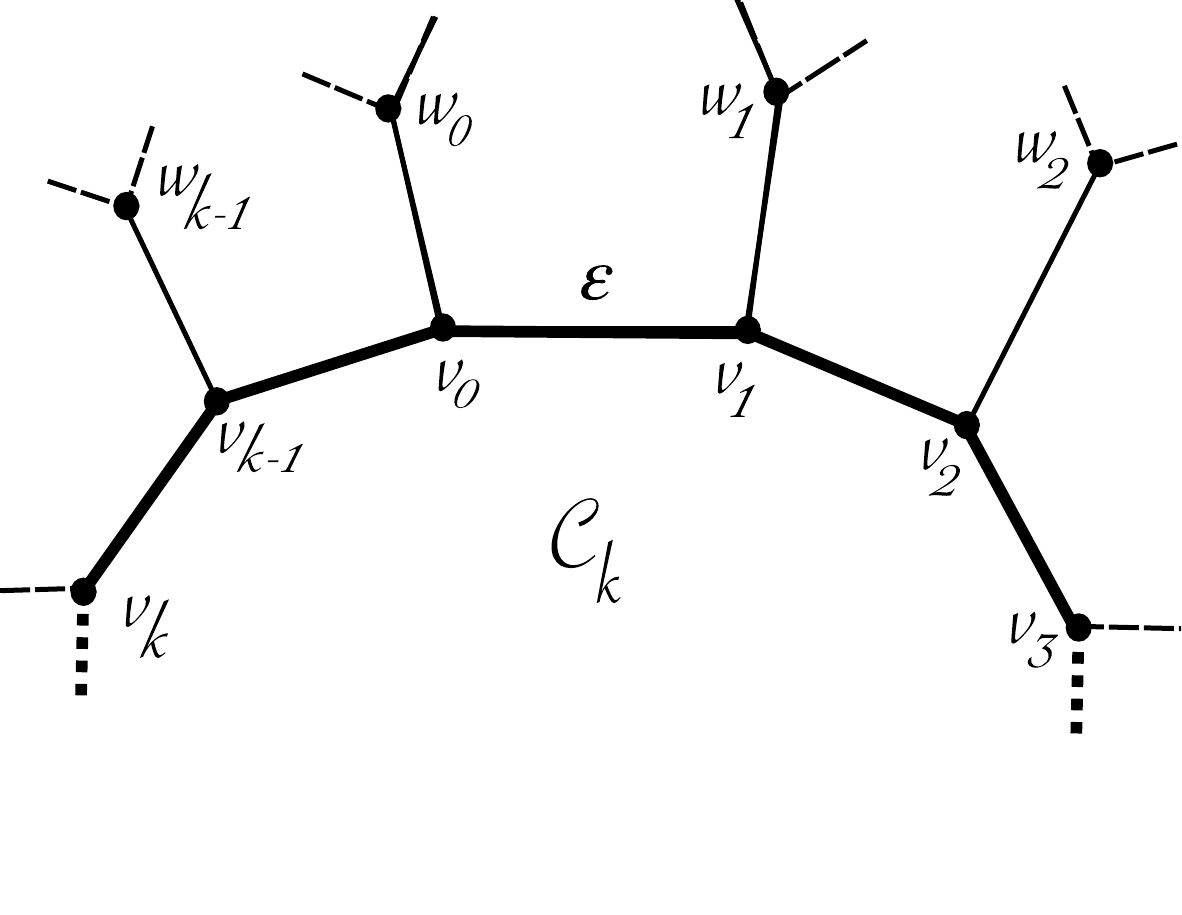}
\caption{The edge $\varepsilon = (v_0, v_1)$ that belongs to a $k$--cycle $C_k$ in $g \in V(\Gamma_n)$ on which a Whitehead move may be performed.} \label{wm-fig:cycle-1}
\end{figure}

Let $g$ be a cubic graph, and let $C_k$ be an indecomposable  $k$--cycle (with $k\geq 2$) in $g$, i.e. an induced cycle subgraph of $g$. %where no edge of $g$ connects any two vertices of $C_k$, except the edges of $C_k$ itself. 
If we perform a Whitehead move on the edge $\varepsilon$ that belongs to a $k$--cycle $C_k$, a part of which is depicted in \Cref{wm-fig:cycle-1}, we obtain a modified graph (partly) depicted in \Cref{wm-fig:cycle-2}.

\begin{figure}[ht]
\centering
\includegraphics[scale=0.3]{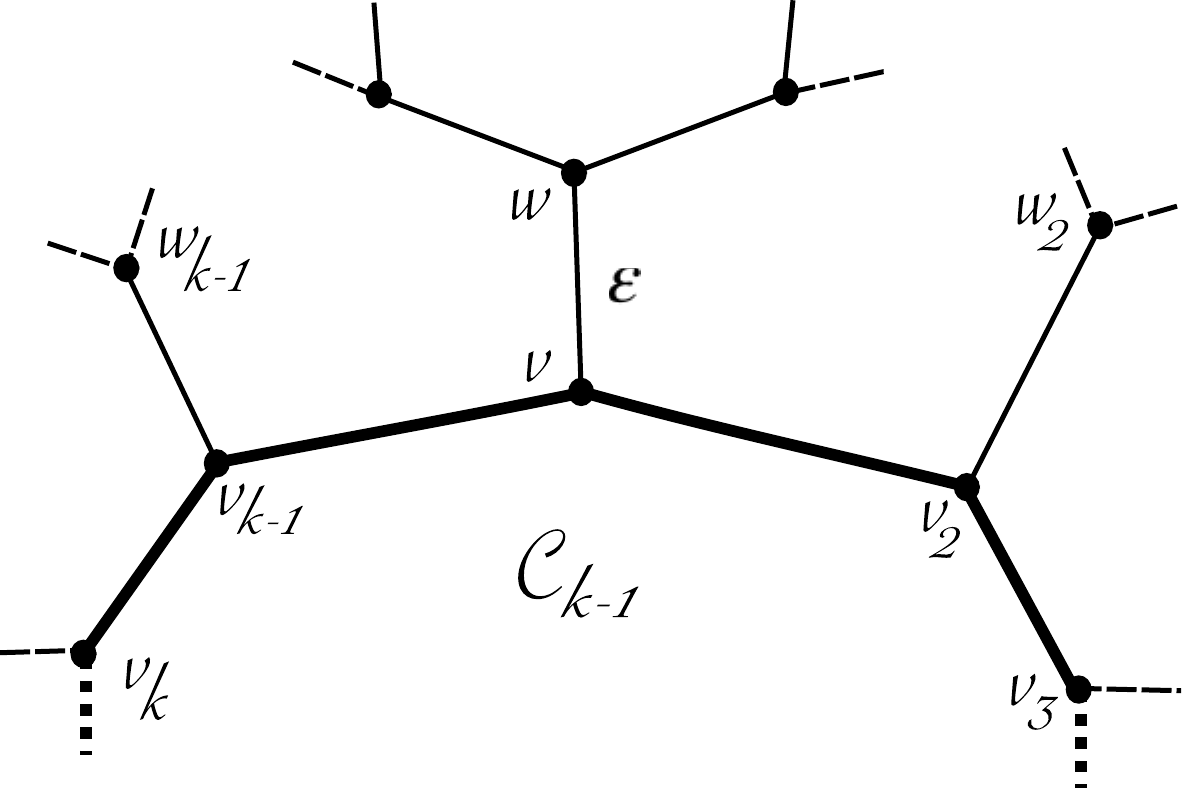}
\caption{ The resulting $(k-1)$--cycle $C_{k-1}$ in $\tilde{g} = w_\varepsilon(g)$: the edge $\varepsilon = (v, w)$ does not belong to $C_{k-1}$. } \label{wm-fig:cycle-2}
\end{figure}

The overall changes in the structure of $\tilde{g} = w_\varepsilon(g)$ as compared to $g$ are local and amount to the following:
\begin{itemize}
\item The $k$--cycle $C_k$ has been transformed into a $(k-1)$--cycle $C_{k-1}$ ;
\item If $v_1$, $w_1$ were part of an $l$--cycle $C_l$ then $C_l$ was transformed into an $(l+1)$--cycle $C_{l+1}$;
\item Same applies to any cycle that previously contained $w_0$ and/or $v_0$.
\end{itemize} 

By performing a total of $k-1$ Whitehead moves on the edges in $C_k$, we reduce it to a loop, as depicted in \Cref{wm-fig:cycle-3}. Here we reduced the total amount of $k$--cycles with $k\geq 2$ by one, although the lengths of some other cycles could have been augmented.

\begin{figure}[ht]
\centering
\includegraphics[scale=0.3]{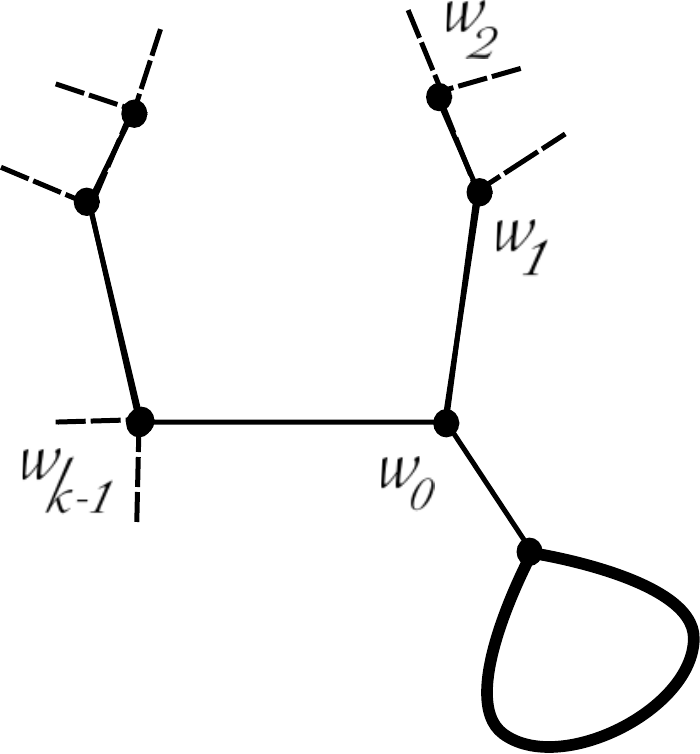}
\caption{ The $k$--cycle $C_k$ ($k \geq 2$) is finally reduced to a loop by $k-1$ consecutive Whitehead moves. } \label{wm-fig:cycle-3}
\end{figure}

By repeating the above procedure on the remaining $k$--cycles (with $k \geq 2$) of the resulting graph, we shall reduce it to a $(n,b)$--lollipop tree, as defined below, an example of which is depicted in \Cref{fig:tree-1}. Such a tree has $2n$ vertices and $b = n+1$ loops: note that $b$ is the $1$st Betty number of the initial graph $g$, since Whitehead moves preserve the number of vertices and edges of $g$, as well as the rank of $\pi_1(g)$.

\begin{figure}[ht]
\centering
\includegraphics[scale=0.3]{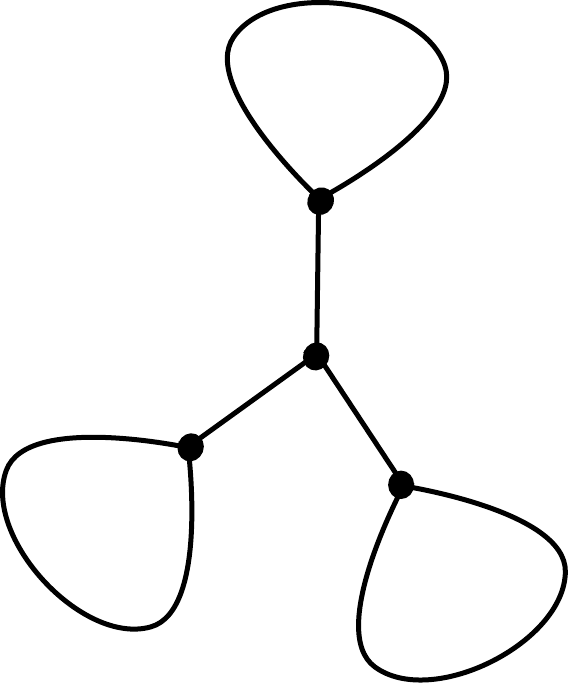}
\caption{ A $(2, 3)$--lollipop tree } \label{fig:tree-1}
\end{figure}

There exists a natural isomorphism between our $(n,b)$--lollipop trees with a chosen loop (i.e. rooted $(n,b)$--lollipop trees with a chosen loop $\varepsilon^*$ as a root), and binary rooted trees. Namely, let us take a rooted $(n,b)$--lollipop tree $(T, \varepsilon^*)$ and remove all the loops as well as the edge $(v^*, v)$ adjacent to $\varepsilon^*$ (with $\varepsilon^*$ being the loop at vertex $v$). The result, $(\tilde{T}, v^*)$ is a binary rooted tree on $2n-1$ vertices with $n$ leaves.

The general theory of binary trees and tree rebalancing \cite[\S 12 -- \S 13]{CLRS} shows that we can bring  $(\tilde{T}, v^*)$ to a unique complete rooted binary tree $T_n$ on $n$ vertices by a series of rotations. As \Cref{fig:tree-2} shows, a tree rotation can be achieved by a Whitehead move. This proves that $\Gamma_n$ is connected.

\begin{figure}[ht]
\centering
\includegraphics[scale=0.35]{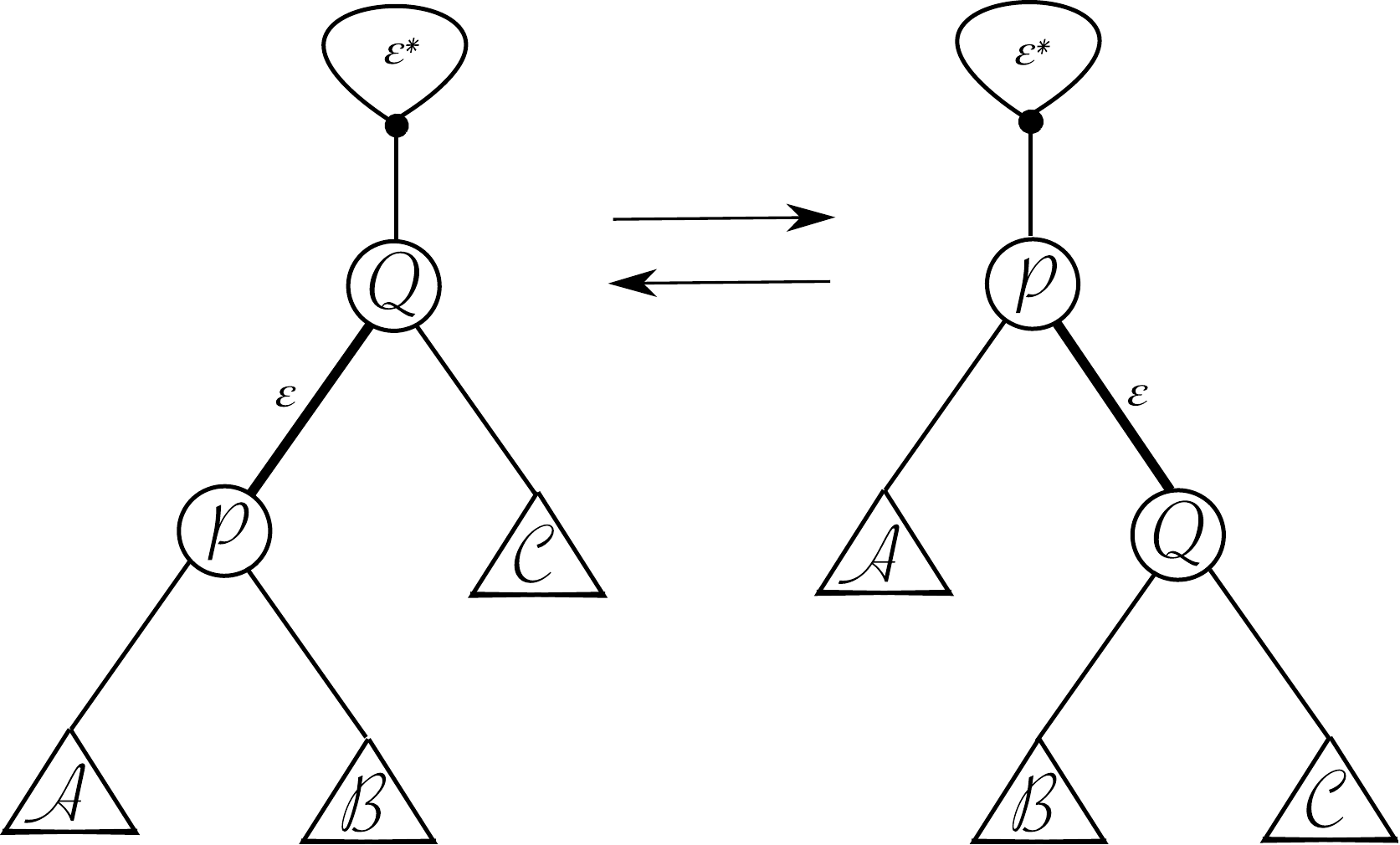}
\caption{A left and right rotation of a rooted binary tree (corresponding to an $(n, b)$--lollipop tree with root $\varepsilon^*$) realised by a Whitehead move on edge $e$. Here $P$, $Q$ are the vertices of $\varepsilon$, and $A$, $B$, and $C$ are the respective sub--trees } \label{fig:tree-2}
\end{figure}
 
\subsection{The outer--conductance of $\Gamma_n$ tends to 0.} 

A \textit{bridge} in a connected graph is an edge such that its removal produces several connected components. A graph with a bridge is called \textit{bridged}, while a graph that has no bridges at all is called \textit{bridgeless}. 

Let $B_n \subset V(\Gamma_n)$  be the set of (isomorphism classes of) connected cubic graphs on $2n$ vertices that have at least one bridge.  We shall estimate the probability $\mathbb{P}(g \in B_n)$ asymptotically. For this matter, notice that if $U_n = V(\Gamma_n)$ is the set of unlabelled cubic graphs (i.e. their isomorphism classes), and $L_n$ is the set of labelled ones, then $\mathbb{P}(g \in U_n \mbox{ has a bridge}) = \mathbb{P}(g \in L_n \mbox{ has a bridge}) + o(1)$, since according to \cite{Bollobas} we have that $|L_n| \sim|U_n| \cdot (2n)!$, as $n\to \infty$.

A \textit{Hamiltonian circuit} in a (connected) graph is a closed edge path that visits each vertex once (i.e. without self--intersection). A graph is called \textit{Hamiltonian} if it admits a Hamiltonian circuit. 

It is clear that a graph that has a bridge cannot be Hamiltonian, and hence every Hamiltonian cubic graph is bridgeless. 

Let ${\rm Bridged} = \{ g \in L_n \ | \ g \mbox{ has a bridge}  \}$ and ${\rm Looped} = \{ g \in L_n \ | \ g \mbox{ has a loop} \}$, and let ${\rm nH}$ be the set of graphs in $L_n$ that do not have a Hamiltonian circuit. Then
\begin{equation*}
\mathbb{P}(g \in {\rm Looped}) \leq \mathbb{P}(g \in {\rm Bridged}) \leq \mathbb{P}(g \in {\rm Looped}) + \mathbb{P}(g \notin {\rm Looped} \mbox{ and } g \in {\rm nH}),
\end{equation*}
since ${\rm Looped} \subseteq {\rm Bridged}$, as well as ${\rm Bridged} \subseteq {\rm nH}$.

By applying \cite[Theorem~9.5]{JLR}, we have that $\mathbb{P}(g \in {\rm Looped}) = 1 - e^{-1} + o(1)$, while \cite[Theorem~9.23]{JLR} implies that $\mathbb{P}(g \notin {\rm Looped} \mbox{ and } g \in {\rm nH}) = o(1)$, as $n\to \infty$. Thus, $\mathbb{P}(g \in {\rm Bridged}) = 1 - e^{-1} + o(1)$ asymptotically, as $n\to \infty$.

Together with the remarks above about the probabilities for labelled and unlabelled graphs, we obtain that $\mathbb{P}(g \in B_n)$, for $g \in V(\Gamma_n)$, satisfies $\mathbb{P}(g \in B_n) = 1 - e^{-1} + o(1)$, as $n\to \infty$. Thus, $|B_n| \sim (1 - e^{-1}) |V(\Gamma_n)|$, and $|V(\Gamma_n)\setminus B_n| \sim e^{-1} |V(\Gamma_n)|$. 

Note that the only vertices in $B_n$ connected with vertices in $V(\Gamma_n)\setminus B_n$ are those $g \in B_n$ having a single bridge, and the corresponding Whitehead moves have to be performed exactly on the bridge edge of $g$. 

Indeed, let $g \in V(\Gamma_n)$ have at least one bridge and let $\varepsilon^*=(u,v)$ denote one of its bridges. Then we can check case--by--case what the result of a Whitehead move on $\varepsilon^*$ could be. 

For this purpose, let us introduce the following equivalence relation on the vertices of $g$: for $v,w \in V(g)$, we write $ v \sim_g w$ if and only if there are two edge--disjoint paths connecting $v$ to $w$. Then we write $[v]_g= \{ w \in V(g) \ s.t. \ v \sim_g w \}$ for the equivalence class of $v$.

Let $\{a,b,c,d\}$ be the vertices adjacent to the bridge $\varepsilon^* \in E(g)$, and let $\tilde{g} = w_{\varepsilon^*}(g)$ be the graph resulting from a Whitehead move performed on $\varepsilon^*$. Then we consider the following five possible cases.  

\begin{enumerate}
\item $[a]_g = [b]_g = [c]_g = [d]_g$: in this case $\varepsilon^*$ is not a bridge neither in $g$, nor in $\tilde{g}$, as shown in \Cref{whitehead-1a}--\Cref{whitehead-1b}.

\begin{figure}[H]
\begin{minipage}[b]{0.45\textwidth}
\centering
\begin{tikzpicture}[scale=0.55]
\vertex (a) at (-1,1)[label=above left:$a$]{};
\vertex (b) at (-1,-1)[label=below left:$b$]{};
\vertex (u) at (0,0)[label=below right:$u$]{};
\vertex (v) at (2,0)[label=below left:$v$]{};
\vertex (c) at (3,1)[label=above right:$c$]{};
\vertex (d) at (3,-1)[label=below right:$d$]{};

\zbimh{(a)}{(c)}{120}{10};
\zbimv{(b)}{(a)}{120}{10};
\zbimv{(c)}{(d)}{120}{10};
\zbimh{(d)}{(b)}{120}{10};

\draw (a) edge [line width=1pt] (u);
\draw (b) edge [line width=1pt] (u);
\draw (u) edge [line width=1pt] node [above, fill=none] {$\varepsilon^*$} (v);
\draw (c) edge [line width=1pt] (v);
\draw (d) edge [line width=1pt] (v);
\end{tikzpicture}
\caption{ Local picture inside $g$ } \label{whitehead-1a}
\end{minipage}
\hfill
\begin{minipage}[b]{0.45\textwidth}
\centering
\begin{tikzpicture}[scale=0.55]
\vertex (a) at (-1,2)[label=above left:$a$]{};
\vertex (b) at (-1,-2)[label=below left:$b$]{};
\vertex (u) at (0,1)[label=below left:$u$]{};
\vertex (v) at (0,-1)[label=above left:$v$]{};
\vertex (c) at (1,2)[label=above right:$c$]{};
\vertex (d) at (1,-2)[label=below right:$d$]{};

\zbimh{(a)}{(c)}{120}{10};
\zbimv{(b)}{(a)}{120}{10};
\zbimv{(c)}{(d)}{120}{10};
\zbimh{(d)}{(b)}{120}{10};

\draw (a) edge [line width=1pt] (u);
\draw (c) edge [line width=1pt] (u);
\draw (v) edge [line width=1pt] node [right, fill=none] {$\varepsilon^*$} (u);
\draw (b) edge [line width=1pt] (v);
\draw (d) edge [line width=1pt] (v);
\end{tikzpicture}
\caption{ Local picture inside $\tilde{g}$ } \label{whitehead-1b}
\end{minipage}
\end{figure}

\item $[a]_g = [b]_g = [c]_g \neq [d]_g$: same as above, $\varepsilon^*$ is not a bridge neither in $g$, nor in $\tilde{g}$, see \Cref{whitehead-2a}--\Cref{whitehead-2b}.

\begin{figure}[H]
\begin{minipage}[b]{0.45\textwidth}
\centering
\begin{tikzpicture}[scale=0.55]
\vertex (a) at (-1,1)[label=above left:$a$]{};
\vertex (b) at (-1,-1)[label=below left:$b$]{};
\vertex (u) at (0,0)[label=below right:$u$]{};
\vertex (v) at (2,0)[label=below left:$v$]{};
\vertex (c) at (3,1)[label=above right:$c$]{};
\vertex (d) at (3,-1)[label=above right:$d$]{};

\zbimh{(a)}{(c)}{120}{10};
\zbimv{(b)}{(a)}{120}{10};
\zbimc{(d)}{-90}{0}{30};

\draw (a) edge [line width=1pt] (u);
\draw (b) edge [line width=1pt] (u);
\draw (u) edge [line width=1pt] node [above, fill=none] {$\varepsilon^*$} (v);
\draw (c) edge [line width=1pt] (v);
\draw (d) edge [line width=1pt] (v);
\end{tikzpicture}
\caption{ Local picture inside $g$ } \label{whitehead-2a}
\end{minipage}
\hfill
\begin{minipage}[b]{0.45\textwidth}
\centering
\begin{tikzpicture}[scale=0.55]
\vertex (a) at (-1,2)[label=above left:$a$]{};
\vertex (b) at (-1,-2)[label=below left:$b$]{};
\vertex (u) at (0,1)[label=below left:$u$]{};
\vertex (v) at (0,-1)[label=above left:$v$]{};
\vertex (c) at (1,2)[label=above right:$c$]{};
\vertex (d) at (1,-2)[label=above right:$d$]{};

\zbimh{(a)}{(c)}{120}{10};
\zbimv{(b)}{(a)}{120}{10};
\zbimc{(d)}{-90}{0}{30};

\draw (a) edge [line width=1pt] (u);
\draw (c) edge [line width=1pt] (u);
\draw (v) edge [line width=1pt] node [right, fill=none] {$\varepsilon^*$} (u);
\draw (b) edge [line width=1pt] (v);
\draw (d) edge [line width=1pt] (v);
\end{tikzpicture}
\caption{ Local picture inside $\tilde{g}$ } \label{whitehead-2b}
\end{minipage}
\end{figure}

\item $[a]_g = [b]_g \neq [c]_g = [d]_g$: in this case, $\varepsilon^*$ is a bridge in $g$, but not in $\tilde{g}$, see \Cref{whitehead-3a}--\Cref{whitehead-3b}.

\begin{figure}[H]
\begin{minipage}[b]{0.45\textwidth}
\centering
\begin{tikzpicture}[scale=0.55]
\vertex (a) at (-1,1)[label=above left:$a$]{};
\vertex (b) at (-1,-1)[label=below left:$b$]{};
\vertex (u) at (0,0)[label=below right:$u$]{};
\vertex (v) at (2,0)[label=below left:$v$]{};
\vertex (c) at (3,1)[label=above right:$c$]{};
\vertex (d) at (3,-1)[label=below left:$d$]{};

\zbimv{(c)}{(d)}{120}{10};
\zbimv{(b)}{(a)}{120}{10};

\draw (a) edge [line width=1pt] (u);
\draw (b) edge [line width=1pt] (u);
\draw (u) edge [line width=1pt] node [above, fill=none] {$\varepsilon^*$} (v);
\draw (c) edge [line width=1pt] (v);
\draw (d) edge [line width=1pt] (v);
\end{tikzpicture}
\caption{ Local picture inside $g$ } \label{whitehead-3a}
\end{minipage}
\hfill
\begin{minipage}[b]{0.45\textwidth}
\centering
\begin{tikzpicture}[scale=0.55]
\vertex (a) at (-1,2)[label=above left:$a$]{};
\vertex (b) at (-1,-2)[label=below left:$b$]{};
\vertex (u) at (0,1)[label=below left:$u$]{};
\vertex (v) at (0,-1)[label=above left:$v$]{};
\vertex (c) at (1,2)[label=above right:$c$]{};
\vertex (d) at (1,-2)[label=below left:$d$]{};

\zbimv{(c)}{(d)}{120}{10};
\zbimv{(b)}{(a)}{120}{10};

\draw (a) edge [line width=1pt] (u);
\draw (c) edge [line width=1pt] (u);
\draw (v) edge [line width=1pt] node [right, fill=none] {$\varepsilon^*$} (u);
\draw (b) edge [line width=1pt] (v);
\draw (d) edge [line width=1pt] (v);
\end{tikzpicture}
\caption{ Local picture inside $\tilde{g}$ } \label{whitehead-3b}
\end{minipage}
\end{figure}

\item   $[a]_g = [b]_g \neq [c]_g$ , $[a]_g = [b]_g \neq [d]_g$, and $[c]_g \neq [d]_g$: again, $\varepsilon^*$ is a bridge in $g$, but not in $\tilde{g}$, see \Cref{whitehead-4a}--\Cref{whitehead-4b}.

\begin{figure}[H]
\begin{minipage}[b]{0.45\textwidth}
\centering
\begin{tikzpicture}[scale=0.55]
\vertex (a) at (-1,1)[label=above left:$a$]{};
\vertex (b) at (-1,-1)[label=below left:$b$]{};
\vertex (u) at (0,0)[label=below right:$u$]{};
\vertex (v) at (2,0)[label=below left:$v$]{};
\vertex (c) at (3,1)[label=above right:$c$]{};
\vertex (d) at (3,-1)[label=below left:$d$]{};

\zbimc{(c)}{-90}{30}{30};
\zbimc{(d)}{-90}{30}{30};
\zbimv{(b)}{(a)}{120}{10};

\draw (a) edge [line width=1pt] (u);
\draw (b) edge [line width=1pt] (u);
\draw (u) edge [line width=1pt] node [above, fill=none] {$\varepsilon^*$} (v);
\draw (c) edge [line width=1pt] (v);
\draw (d) edge [line width=1pt] (v);
\end{tikzpicture}
\caption{ Local picture inside $g$ } \label{whitehead-4a}
\end{minipage}
\hfill
\begin{minipage}[b]{0.45\textwidth}
\centering
\begin{tikzpicture}[scale=0.55]
\vertex (a) at (-1,2)[label=above left:$a$]{};
\vertex (b) at (-1,-2)[label=below left:$b$]{};
\vertex (u) at (0,1)[label=below left:$u$]{};
\vertex (v) at (0,-1)[label=above left:$v$]{};
\vertex (c) at (1,2)[label=above right:$c$]{};
\vertex (d) at (1,-2)[label=below left:$d$]{};

\zbimc{(c)}{-90}{30}{30};
\zbimc{(d)}{-90}{30}{30};
\zbimv{(b)}{(a)}{120}{10};

\draw (a) edge [line width=1pt] (u);
\draw (c) edge [line width=1pt] (u);
\draw (v) edge [line width=1pt] node [right, fill=none] {$\varepsilon^*$} (u);
\draw (b) edge [line width=1pt] (v);
\draw (d) edge [line width=1pt] (v);
\end{tikzpicture}
\caption{ Local picture inside $\tilde{g}$ } \label{whitehead-4b}
\end{minipage}
\end{figure}

\item  $[a]_g , [b]_g , [c]_g , [d]_g$ are all distinct: obviously, $\varepsilon^*$ is a bridge in both $g$ and $\tilde{g}$, see \Cref{whitehead-5a}--\Cref{whitehead-5b}.

\begin{figure}[H]
\begin{minipage}[t]{0.45\textwidth}
\centering
\begin{tikzpicture}[scale=0.55]
\vertex (a) at (-1,1)[label=above right:$a$]{};
\vertex (b) at (-1,-1)[label=below right:$b$]{};
\vertex (u) at (0,0)[label=below right:$u$]{};
\vertex (v) at (2,0)[label=below left:$v$]{};
\vertex (c) at (3,1)[label=above right:$c$]{};
\vertex (d) at (3,-1)[label=below left:$d$]{};

\zbimc{(c)}{-90}{30}{30};
\zbimc{(d)}{-90}{30}{30};
\zbimc{(a)}{90}{180}{30};
\zbimc{(b)}{180}{-90}{30};

\draw (a) edge [line width=1pt] (u);
\draw (b) edge [line width=1pt] (u);
\draw (u) edge [line width=1pt] node [above, fill=none] {$\varepsilon^*$} (v);
\draw (c) edge [line width=1pt] (v);
\draw (d) edge [line width=1pt] (v);
\end{tikzpicture}
\caption{ Local picture inside $g$  } \label{whitehead-5a}
\end{minipage}
\hfill
\begin{minipage}[t]{0.45\textwidth}
\centering
\begin{tikzpicture}[scale=0.55]
\vertex (a) at (-1,2)[label=above right:$a$]{};
\vertex (b) at (-1,-2)[label=below right:$b$]{};
\vertex (u) at (0,1)[label=below left:$u$]{};
\vertex (v) at (0,-1)[label=above left:$v$]{};
\vertex (c) at (1,2)[label=above right:$c$]{};
\vertex (d) at (1,-2)[label=below left:$d$]{};

\zbimc{(c)}{-90}{30}{30};
\zbimc{(d)}{-90}{30}{30};
\zbimc{(a)}{90}{180}{30};
\zbimc{(b)}{180}{-90}{30};

\draw (a) edge [line width=1pt] (u);
\draw (c) edge [line width=1pt] (u);
\draw (v) edge [line width=1pt] node [right, fill=none] {$\varepsilon^*$} (u);
\draw (b) edge [line width=1pt] (v);
\draw (d) edge [line width=1pt] (v);
\end{tikzpicture}
\caption{Local picture inside $\tilde{g}$ } \label{whitehead-5b}
\end{minipage}
\end{figure}

\end{enumerate}
    
According to the observations above, if a bridged graph $g \in B_n$ can be transformed into a bridgeless one $\tilde{g} = w_\varepsilon(g)$, then $\varepsilon$ is the only bridge of $g$. Every time, both Whitehead moves either succeed or fail to bring $\tilde{g}$ outside $B_n$. This implies that $|\partial(B_n)| \leq 2 |B_n|$. Thus the outer--conductance of the set $B_n$ satisfies 
\begin{align*}
\phi_{\rm out}(B_n) &= \frac{|\partial(B_n)|}{\min\{ d(B_n), d(V(\Gamma_n)\setminus B_n) \}} \leq \frac{2 |B_n|}{ 6n |V(\Gamma_n)\setminus B_n|} \sim \\
&\sim \frac{2 |V(\Gamma_n)| (1-e^{-1})}{6n |V(\Gamma_n)| e^{-1}} = \frac{e-1}{3n} \to 0,
\end{align*} 
as $n\to \infty$. \hfill $\square$

\begin{remark}
It appears that $\Gamma_n$ is itself a relatively ``thin'' neighbourhood of $B_n$. Indeed, each length $k$ cycle in a graph $g \in V(\Gamma_n)$ can be reduced by $k-2$ Whitehead moves to a $2$--cycle, and one more Whitehead move will create an edge incident to a loop. The standard asymptotic bound on girth is $\mathrm{girth}(g) \lesssim 2 \log_2 n$. If $v = |V(\Gamma_n)|$, then $v \sim n^n$ \cite{Bollobas}, and thus each vertex of $\Gamma_n$ happens to be within $\sim \log_2 \log_2 v$ distance from $B_n$. As shown in \cite{RafiTao}, the diameter of $\Gamma_n$ satisfies $d = \mathrm{diam}(\Gamma_n) \sim n \log_2 n$, and thus each vertex of $\Gamma_n$ is within $\log_2 d$ of $B_n$.
\end{remark}

\section{Combinatorics of paths}

Let $A$ be a non--zero square $k\times k$ ($k\geq 2$) matrix with non--negative integer entries. Such a matrix is called \textit{non--negative}, for brevity. If all the entries of $A$ are actually positive, then $A$ is called \textit{positive}.

If there exists a permutation matrix $P$ such that $PAP^{-1}$ has upper--triangular block form, $A$ is called \textit{reducible}. Otherwise, $A$ is \textit{irreducible}.

For $i \in \{1, \dots, k\}$, the \textit{$i$--th period} of $A$ is the greatest common divisor of all integers $m$ such that the $i$--th diagonal entry of $A^m$ is positive. If $A$ is irreducible, then all its periods coincide, and each of them equals \textit{the period} of $A$. If the period of $A$ equals $1$, then $A$ is called \textit{aperiodic}. 

An aperiodic and irreducible non--negative matrix $A$ is called \textit{primitive}. If $A$ is a primitive matrix, then its spectral radius is a Perron number, as a consequence of the Perron--Frobenius theorem \cite[Theorem 4.5.11]{LM}. Recall that a \textit{Perron number} is an algebraic integer $\rho > 1$ such that all of its other Galois conjugates have modulus strictly less than $\rho$. Perron numbers often appear in dynamical context, cf. \cite{LM}. Another property of the spectral radius of a primitive matrix $A$ is that this eigenvalue has multiplicity one.  

Let $\Gamma$ be a directed graph, and let $A$ be its adjacency matrix defined as follows:
\begin{equation*}
A_{uv} = \mbox{ the number of directed edges joining } u \mbox{ to } v, \mbox{ if } u\neq v, \mbox{ and }
\end{equation*}
\begin{equation*}
A_{vv} = \mbox{ the number of loops incident to } v.
\end{equation*}

Then $A$ is irreducible if and only if $\Gamma$ is \textit{strongly connected}, i.e. there exists a path of directed edges in $\Gamma$ between any pair of distinct vertices. If $\Gamma$ is strongly connected and the greatest common divisor of its directed cycle lengths equals $1$, then $A$ is aperiodic, and thus primitive. Thus, in this case we can deduce that the spectral radius of $A$ is a Perron number by considering the combinatorics of $\Gamma$.

\begin{corollary}\label{cor:graph-of-graphs-1}
The graph $\Gamma_n$ is strongly connected and aperiodic. The number of length $\ell$ paths in $\Gamma_n$ is asymptotic to ${\rm const} \cdot \rho^\ell_n$, with $\rho_n$ a Perron number, if $\ell \to \infty$.  
\end{corollary}

\begin{proof}
For every $n\geq 1$, there exists an $(n, b)$--lollipop tree $t\in V(\Gamma_n)$, with $b = n+1 \geq 2$ loops $\varepsilon_1, \dots, \varepsilon_b$. Then, $w_{\varepsilon_i}(t) = t$, for any $i=1, \dots, b$, and thus $\Gamma_n$ has loops. This implies that $\Gamma_n$ is aperiodic.

\begin{remark}
    There exist a $3$--cycle and a $4$--cycle in each $\Gamma_n$, for $n\geq 3$. Namely, we can choose the cycle in \Cref{fig:coprime-cycles}. The subgraph of $\Gamma_3$ depicted in \Cref{fig:coprime-cycles} embeds in $\Gamma_n$ by replacing the left-most loop in each cubic graph with a sequence of ``lollipop'' graphs.  
\end{remark}

\begin{figure}[ht]
\centering

\begin{tikzpicture}[scale=0.1]
  % Define styles
  \tikzset{vertex/.style={rectangle, draw, align=center, inner sep=2pt}}

  % Define vertices for the 3-cycle with placeholders
  \node[vertex] (a) at (0, 0) {\includegraphics[scale=0.25]{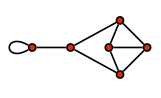}};
  \node[vertex] (b) at (30, 0) {\includegraphics[scale=0.25]{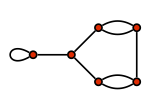}};
  \node[vertex] (c) at (15, 26) {\includegraphics[scale=0.25]{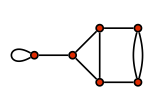}}; % Shared vertex
  
  % Define vertices for the 4-cycle with placeholders
  \node[vertex] (d) at (15, 52) {\includegraphics[scale=0.25]{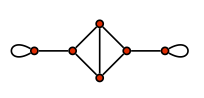}};
  \node[vertex] (e) at (-15, 52) {\includegraphics[scale=0.25]{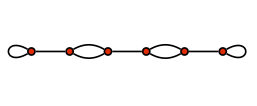}};
  \node[vertex] (f) at (-15, 26) {\includegraphics[scale=0.25]{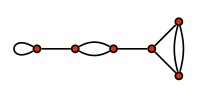}};
  
 % Draw the 3-cycle with arrows
  \draw[-latex] (a) -- (b);
  \draw[-latex] (b) -- (c);
  \draw[-latex] (c) -- (a);
  
  % Draw the 4-cycle with arrows
  \draw[-latex] (c) -- (d);
  \draw[-latex] (d) -- (e);
  \draw[-latex] (e) -- (f);
  \draw[-latex] (f) -- (c);

\end{tikzpicture}

\caption{Two cycles in $\Gamma_3$ of coprime lengths having a common vertex. Each of them is a cubic graph containing a ``lollipop'' (a loop incident to an edge) subgraph on the left. By replacing this lollipop subgraph with a lollipop tree on $2n-4$ vertices minus a loop, we obtain a similar picture inside $\Gamma_n$ for every $n\geq 4$ as well.}\label{fig:coprime-cycles}
\end{figure}

Moreover $\Gamma_n$ is strongly connected, as follows from the proof of Theorem~\ref{wm-thm:graph-of-graphs}. Thus, we obtain that $A_n$ is  primitive, and can now apply the Perron--Frobenius theorem. 

Hence the spectral radius $\rho_n = \rho(A_n)$ is a multiplicity $1$ eigenvalue with a positive eigenvector $v$, and $\rho_n$ is a Perron number. Then the asymptotic number of paths follows from \cite[Theorem 4.5.12]{LM} by a standard computation. 
\end{proof}

\section{Related objects}\label{wm-section:related}

Above we consider the graph of graphs $\Gamma_n$ of all possible cubic graphs on $2n$ vertices, where two vertices, $g_1, g_2 \in V(\Gamma_n)$, are connected by a directed edge for each Whitehead move $w_\varepsilon$, for $\varepsilon \in E(g_1)$, such that $w_\varepsilon(g_1)$ is isomorphic to $g_2$.

However, there are several ways to modify this definition. Indeed, we could modify the edges of $\Gamma_n$ in order to turn it into an undirected graph, a simple graph or other types of graphs. These modifications can provide relations between the graph of graphs and other mathematical problems. In this section we will explore some of these connections.

\subsection{Matchings on graphs}\label{wm-related:matchings}

Let us consider two edges $e$, $f$ in $g \in \Gamma_n$ that are not incident to each other. Note that the Whitehead move $w^{(k)}_e(g)$, $k=1,2$, only affects the $4$ half--edges incident to the vertices of $e$ (and not belonging to $e$). The same holds for $w^{(k)}_f(g)$, $k=1,2$. Thus, the respective sets of half--edges are disjoint, which implies that $w_e(g)$ and $w_f(g)$ commute. We will call this transformation a \emph{simultaneous Whitehead move on the set $\{e, f\}$}, cf. \cite{RafiTao}.

Thus, we can consider a modified graph $\Gamma'_n$ with $V(\Gamma'_n) = V(\Gamma_n)$ where $g_1, g_2 \in V(\Gamma'_n)$ are connected by a directed edge if there exists a Whitehead move on some set of non--incident edges of $g_1$ mapping $g_1$ to $g_2$.

Each vertex matching of cardinality $m\geq 1$ in $g_1$ gives rise to $2^m$ simultaneous Whitehead moves. Thus, knowing the number of cardinality $m\geq 1$ matchings in $g_1$ allows us to compute the degree of $g_1$ in $\Gamma'_n$. 

\subsection{Outer space and its spine}\label{wm-related:out}

Let $F_n$ be the free group on $n \geq 2$ letters. In \cite{culler1986moduli}, Culler and Vogtmann introduced a space $X_n$ on which the group $Out(F_n)$ acts. This space can be thought of as analogous to the Teichm\"uller space of a surface with the action of the mapping class group of the surface. 

The description of $X_n$ is complicated, and we will not detail it here. We will only briefly mention the relevant aspects of this object and how it could relate to our graph of graphs.

The space $X_n$ is an infinite, finite--dimensional simplicial complex. A point in $X_n$ is defined as a metric graph (where edges have assigned lengths) of total length $1$ with fundamental group $F_n$ equipped with a homotopy equivalence to the bouquet of $n$ circles, called a \emph{marking}. Each open simplex in $X_n$ is obtained by varying the lengths of the edges in a fixed marked topological graph in such a way that their sum remains equal to $1$. The group $\mathrm{Out}(F_n)$ acts on $X_n$ by change of markings.

Two given $k$-simplices $x_1$ and $x_2$ share a face of codimension 1 if one can pass from a point in $x_1$ to a point in $x_2$ by collapsing an edge and re-opening it, i.e., by a continuous version of a Whitehead move. Observe that some simplices have missing faces corresponding to loops in the graph, since collapsing a loop would change the fundamental group. Note that simplices of maximal dimension correspond to marked trivalent metric graphs.

The space $X_n$ admits a \emph{spine} $K_n$ which is a deformation retract of $X_n$: the vertices of $K_n$ are marked graphs considered without edge lengths. This spine $K_n$ has the structure of a simplicial complex, in fact it can be identified with the geometric realization of the partially ordered set of open simplices of $X_n$\footnote{This is, word for word, the description given in \cite{vogtmann2002automorphisms}}. The group $Out(F_n)$ acts on $K_n$ with compact quotient. 

Our graph of graphs $\Gamma_n$ is related to $K_{n+1}/Out(F_{n+1})$ in the following way. Let us consider the dual graph to the collection of top--dimensional simplices in $X_{n+1}$, then take its quotient by $Out(F_{n+1})$. This will yield a graph that $\Gamma_n$ admits a morphism onto.

\section*{Acknowledgements}
LGdP was supported by the FWO and the F.R.S.--FNRS under the Excellence of Science (EOS) programme project ID~40007542. Both authors thank the anonymous referees for their comments and suggestions that helped improving the manuscript. 

{
\renewcommand*\MakeUppercase[1]{#1}
\bibliographystyle{alpha}							%Style alpha : les références sont notées avec nom de l'auteur et année.
\bibliography{biblio}						%On appelle la bibliographie
}
\end{document}